\documentclass[a4paper]{amsart}
\usepackage[utf8]{inputenc}
\usepackage[T1]{fontenc}
\usepackage{amsfonts}
\usepackage{amsmath}
\usepackage{amssymb}
\usepackage{amstext}
\usepackage{amsthm}
\usepackage[shortlabels]{enumitem}
\usepackage{geometry}
\usepackage{fullpage}
\usepackage{hyperref}
\usepackage{cleveref}
\usepackage{mathtools}
\usepackage{parskip}
\usepackage{tikz}
\usepackage{xcolor}
\usepackage{shuffle}
\usepackage{ifthen}
\usepackage{footmisc}

\usepackage{algorithm}
\usepackage{algpseudocode}

\algrenewcommand\algorithmicrequire{\textbf{Input:}}
\algrenewcommand\algorithmicensure{\textbf{Output:}}

\newtheorem{theorem}{Theorem}[section]
\newtheorem{lemma}[theorem]{Lemma}

\newtheorem{proposition}[theorem]{Proposition}

\theoremstyle{definition}
\newtheorem{definition}[theorem]{Definition}
\newtheorem{example}[theorem]{Example}

\theoremstyle{remark}
\newtheorem{remark}[theorem]{Remark}

\numberwithin{equation}{section}

\title[Shuffle theorem and sandpiles]{Shuffle theorems and sandpiles}

\author{Michele D'Adderio}
\address{Universit\`a di Pisa\\Dipartimento di Matematica\\ Largo Bruno Pontecorvo 5, 56127 Pisa\\ Italy}\email{michele.dadderio@unipi.it}

\author{Mark Dukes}
\address{University College Dublin\\ School of Mathematics and Statistics\\  Belfield, Dublin 4\\ Ireland}\email{mark.dukes@ucd.ie}

\author{Alessandro Iraci}
\address{Universit\`a di Pisa\\Dipartimento di Matematica\\ Largo Bruno Pontecorvo 5, 56127 Pisa\\ Italy}\email{alessandro.iraci@unipi.it}

\author{Alexander Lazar}
\address{Département de Mathématique\\ Université Libre de Bruxelles\\ Bruxelles, 1050\\ Belgique}\email{alexander.leo.lazar@ulb.be}

\author{Yvan Le Borgne}
\address{LaBRI\\ Université Bordeaux 1\\ 33405 Talence cedex\\ France}\email{yvan.le-borgne@u-bordeaux.fr}

\author{Anna Vanden Wyngaerd}
\address{Département de Mathématique\\ Université Libre de Bruxelles\\ Bruxelles, 1050\\ Belgique}\email{anna.vanden.wyngaerd@ulb.be}

\newcommand{\N}{\mathbb{N}}
\newcommand{\Z}{\mathbb{Z}}

\newcommand{\area}{\mathsf{area}}
\newcommand{\dinv}{\mathsf{dinv}}
\newcommand{\pmaj}{\mathsf{pmaj}}
\newcommand{\lev}{\mathsf{level}}
\newcommand{\del}{\mathsf{delay}}
\newcommand{\maj}{\mathsf{maj}}

\newcommand{\PF}{\mathsf{PF}}

\renewcommand{\th}{^\mathsf{th}}

\begin{document}

\begin{abstract}
	We provide an explicit description of the recurrent configurations of the sandpile model on a family of graphs $\widehat{G}_{\mu,\nu}$, which we call \emph{clique-independent} graphs, indexed by two compositions $\mu$ and $\nu$. Moreover, we define a \emph{delay} statistic on these configurations, and we show that, together with the usual \emph{level} statistic, it can be used to provide a new combinatorial interpretation of the celebrated \emph{shuffle theorem} of Carlsson and Mellit. More precisely, we will see how to interpret the polynomials $\langle \nabla e_n, e_\mu h_\nu \rangle$ in terms of these configurations.
\end{abstract}

\maketitle

\section{Introduction}


The celebrated \emph{shuffle theorem} of Carlsson and Mellit \cite{CarlssonMellitShuffle} provided a positive solution to a long-standing conjecture about a combinatorial formula for the Frobenius characteristic of the so-called diagonal harmonics. More precisely, this theorem gives the monomial expansion of the symmetric function $\nabla e_n$, where $e_n$ is the elementary symmetric function of degree $n$ in the variables $x_1,x_2,\dots$, and $\nabla$ is the famous \emph{nabla} operator introduced in \cite{Bergeron-Garsia-Haiman-Tesler-Positivity-1999}. In this formula, to each \emph{labelled Dyck path} of size $n$ corresponds a monomial, where the variables $x_1,x_2,\dots$ keep track of the labels, while the variables $q$ and $t$ keep track of the bistatistic ($\dinv$, $\area$).

In \cite{LoehrRemmel_pmaj} Loehr and Remmel provided an alternative combinatorial interpretation of the same symmetric function in terms of the same objects, but using the bistatistic ($\area$, $\pmaj$). In particular, they showed bijectively that the two combinatorial formulas coincide. In the present article we show that this last combinatorial formula has a natural interpretation in terms of the sandpile model.

\medskip


The \emph{(abelian) sandpile model} is a combinatorial dynamical system on graphs first introduced by Bak, Tang and Wiesenfeld \cite{BakTangWiesenfeld} in the context of ``self-organized criticality'' in statistical mechanics. The sandpile model (and variants of it) have found applications in a wide variety of mathematical contexts including enumerative combinatorics, tropical geometry, and Brill--Noether theory, among others: see \cite{Klivans_Sandpiles} for a nice introductory monograph. \textit{In the present article we only consider the sandpile model with a sink}.

A well-known link between the combinatorics of this dynamical system and that of the underlying graph is given by the so-called \emph{recurrent configurations} (see \Cref{def:recurrent}). For example, the recurrent configurations of the sandpile model are in bijection with the spanning trees of the graph (see e.g.\ \cite{CoriLeBorgne_Tutte}).

If the underlying graph presents some symmetries, then it is natural to look at the recurrent configurations modulo those symmetries. For example, for the complete graph we can identify recurrent configurations that are the same up to a permutation of the vertices (not moving the sink): perhaps not surprisingly, we still unearth some interesting combinatorics, as in this case we find Catalan many such sorted configurations.

More formally, consider the sandpile model on a graph $G$, and let $\mathsf{Aut}(G)$ be the automorphism group of $G$. Consider a subgroup $\Gamma$ of the stabilizer of the sink. Now $\Gamma$ acts naturally on the set $\mathsf{Rec}(G)$ of recurrent configurations: we are interested in the orbits of this action that we will call \emph{sorted recurrent configurations}.

\medskip


We will consider an explicit family of graphs $\widehat{G}_{\mu,\nu}$ indexed by pairs of compositions $\mu$ and $\nu$. For such a graph $\widehat{G}_{\mu,\nu}$ we will look at a subgroup $\Gamma$ of its automorphism group that will be isomorphic to the Young subgroup $\mathfrak{S}_\mu\times \mathfrak{S}_\nu$ of the symmetric group $\mathfrak{S}_n$, where $n=\lvert \mu \rvert + \lvert \nu \rvert$. We denote by $\mathsf{SortRec}(\mu;\nu)$ the set of corresponding sorted recurrent configurations of $\widehat{G}_{\mu,\nu}$.

For every recurrent configuration $\kappa$ of $\widehat{G}_{\mu,\nu}$, we will define a new statistic, called the \emph{delay} of $\kappa$ (denoted $\del(\kappa)$), which we will couple with the usual \emph{level} statistic (denoted $\lev(\kappa)$). To state our main result, we need a few more definitions.

Given a composition $\mu=(\mu_1,\mu_2,\dots)$, we denote by $e_\mu$ the product $e_{\mu_1}e_{\mu_2}\cdots$, and similarly $h_\mu=h_{\mu_1}h_{\mu_2}\cdots$, where $e_n$ and $h_n$ are the elementary and complete homogeneous symmetric function of degree $n$, respectively. Finally, we denote by $\langle - ,-\rangle$ the usual (Hall) scalar product on symmetric functions (see \cite[Chapter~1]{Haglund-Book-2008}).
\begin{theorem}
	\label{thm:main}
	For every pair of compositions $\mu$, $\nu$ such that $n = \lvert \mu \rvert + \lvert \nu \rvert$ we have
	\[\langle \nabla e_n, e_\mu h_\nu \rangle =\sum_{\kappa\in \mathsf{SortRec}(\mu;\nu)}q^{\lev(\kappa)}t^{\del(\kappa)}.\]
\end{theorem}
Notice that for $\mu=\varnothing$, the coefficient $\langle \nabla e_n, h_\nu  \rangle$ is simply the coefficient of $x^\nu=x_1^{\nu_1}x_2^{\nu_2}\cdots$ in $\nabla e_n$, hence this formula gives in particular a new combinatorial interpretation of the monomial expansion of the symmetric function $\nabla e_n$ in terms of the sandpile model.

The idea of the proof is to show that the sorted recurrent configurations with the ($\lev$, $\del$) bistatistic correspond bijectively to the labelled Dyck paths predicted by the shuffle theorem with the ($\area$, $\pmaj$) bistatistic.

\medskip


Notice that \Cref{thm:main} extends several previous results in the literature: the case $\widehat{G}_{\varnothing,(n)}$ was already worked out in \cite{CoriLeBorgne_Kn}, (a slight modification of) the case $\widehat{G}_{(m,n-m),\varnothing}$ already appears in \cite{ADDHL_Narayana,DukesLeBorgne_Kmn}, while the case $\widehat{G}_{(m),(n-m)}$ is dealt with in \cite{DeryckeDukesLeBorgneSplitgraphs}.

Other articles in which sorted recurrent configurations appear are for example \cite{ADDL_two_op} and \cite{DAdderioLeBorgne_Kmn}. It should be noticed that the works \cite{CoriLeBorgne_Kn} and \cite{DAdderioLeBorgne_Kmn} inspired the results in \cite{CoolsPanizzut_Kn} and \cite{CDJP_Kmn} respectively, which belong to tropical geometry and Brill-Noether theory.

We hope that the findings in the present article motivate further investigation of sorted recurrent configurations, and their relation to other parts of mathematics.

\medskip


The article is organized in the following way. In \Cref{sec:shufflethm} we recall the combinatorial definitions of the shuffle theorem with the ($\area$, $\pmaj$) bistatistic. In \Cref{sec:Gmunu} we introduce our clique-independent graphs, while in \Cref{sec:basics-sandpile} we recall some basic facts about the sandpile model. In \Cref{sec:algo} we introduce the statistic $\del$ via a toppling algorithm, which will be used in \Cref{sec:SortRec} to characterize the sorted recurrent configurations of our clique-independent graphs. Finally, in \Cref{sec:bijection} we provide a bijection between sorted configurations and parking functions, proving our main result.

\subsection*{Acknowledgments}

D'Adderio is partially supported by PRIN 2022A7L229 ALTOP and by INDAM research group GNAMPA. D'Adderio, Lazar and Vanden Wyngaerd are partially supported by ARC “From algebra to combinatorics, and back”. Le Borgne is partially supported by ANR Combin\'e ANR-19-CE48-0011.

\section{Combinatorics of the shuffle theorem}\label{sec:shufflethm}

For every $n\in \N$, we set $[n] \coloneqq \{1,2,\dots,n\}$. The pmaj statistic was first introduced in \cite{LoehrRemmel_pmaj}. The area statistic is the classical area statistic that has appeared in the many other papers concerning this topic.

\begin{definition}
	A \emph{Dyck path} of size $n$ is a lattice path going from $(0,0)$ to $(n,n)$, using only north and east steps and staying weakly above the line $x=y$ (also called the \emph{main diagonal}). A \emph{labelled Dyck path} is a Dyck path whose vertical steps are labelled with (not necessarily distinct) positive integers such that, when placing the labels in the square to the right of its step, the labels appearing in each column are strictly increasing from bottom to top. For us, a \emph{parking function}\footnote{These are in bijection with the functions $f \colon [n] \rightarrow [n]$ such that $ \# \{ 1 \leq j \leq n \mid f(j) \geq i \} \leq n+1-i$, by defining $f(i)$ to be the column containing the label $i$.} of size $n$ is a labelled Dyck path of size $n$ whose labels are precisely the elements of $[n]$. See \Cref{fig:LDP_config} for an example.
\end{definition}

The set of all parking functions of size $n$ is denoted by $\PF(n)$.

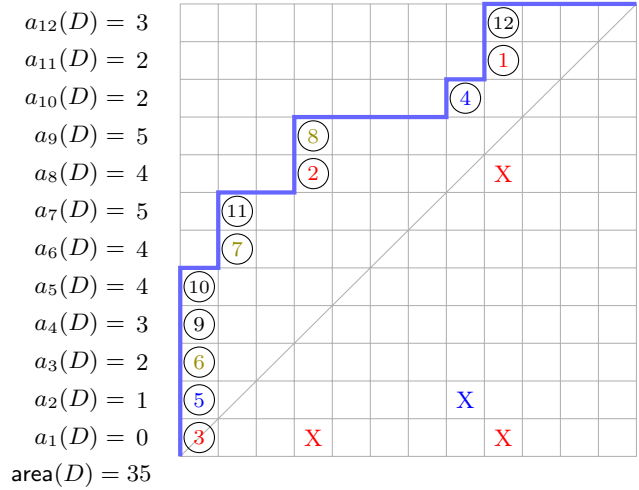
\begin{figure}[!ht]
	\centering

	\begin{tikzpicture}[scale=0.5]
		\draw[gray!60, thin] (0,0) grid (12,12) (0,0) -- (12,12);
		\draw[blue!60, line width = 1.6pt] (0,0)|-(1,5)|-(3,7)|-(7,9)|-(8,10)|-(12,12);

		\draw
		(0.5,0.5) circle(0.4 cm) node {\textcolor{red}{\footnotesize{$3$}}}
		(0.5,1.5) circle(0.4 cm) node {\textcolor{blue}{\footnotesize{$5$}}}
		(0.5,2.5) circle(0.4 cm) node {\textcolor{olive}{\footnotesize{$6$}}}
		(0.5,3.5) circle(0.4 cm) node {\footnotesize{$9$}}
		(0.5,4.5) circle(0.4 cm) node {\scriptsize{$10$}}
		(1.5,5.5) circle(0.4 cm) node {\textcolor{olive}{\footnotesize{$7$}}}
		(1.5,6.5) circle(0.4 cm) node {\scriptsize{$11$}}
		(3.5,7.5) circle(0.4 cm) node {\textcolor{red}{\footnotesize{$2$}}}
		(3.5,8.5) circle(0.4 cm) node {\textcolor{olive}{\footnotesize{$8$}}}
		(7.5,9.5) circle(0.4 cm) node {\textcolor{blue}{\footnotesize{$4$}}}
		(8.5,10.5) circle(0.4 cm) node {\textcolor{red}{\footnotesize{$1$}}}
		(8.5,11.5) circle(0.4 cm) node {\scriptsize{$12$}}
		(3.5,0.5) node {\textcolor{red}{\small{X}}}
		(8.5,0.5) node {\textcolor{red}{\small{X}}}
		(8.5,7.5) node {\textcolor{red}{\small{X}}}
		(7.5,1.5) node {\textcolor{blue}{\small{X}}};
		\foreach \i in {1,...,12}{
				\draw node[anchor=east] at (-1.2,\i-0.5) {\small $a_{\i}(D)=$};
			}
		\foreach \i/\j in {1/0,2/1,3/2,4/3,5/4,6/4,7/5,8/4,9/5,10/2,11/2,12/3}{
				\draw node at (-1,\i-0.5) {\small$\j$};
			}
		\draw node[anchor=east] at (-0.5,-0.5) {\small $\area(D) = 35$};
	\end{tikzpicture}

	\caption{An element $D$ of $\PF((4,3);(3,2))$. The colors and crosses are explained in Example~\ref{ex:bijection}.}
	\label{fig:LDP_config}
\end{figure}

\begin{definition}
	Given $D\in \PF(n)$, we define its \emph{area word} to be the string $a(D) = a_1(D) \cdots a_n(D)$ where $a_i(D)$ is the number of whole squares in the $i\th$ row (from the bottom) between the path and the main diagonal. We define the \emph{area} of $D$ as \[ \area(D) \coloneqq \sum_{i=1}^n a_i(D). \]
\end{definition}

\begin{example}
	\label{ex:area}
	The area word of the path in \Cref{fig:LDP_config} is $012344545223$ and its area is $35$.
\end{example}

To introduce the other statistic, we need a couple of definitions.

\begin{definition}
	Let $a_1 a_2 \cdots a_k$ be a word of integers. We define its \emph{descent set} \[ \mathsf{Des}(a_1 a_2 \cdots a_k) \coloneqq \{ 1 \leq i \leq k-1 \mid a_i > a_{i+1} \} \] and its \emph{major index} $\mathsf{maj}(a_1 a_2 \cdots a_k)$ as the sum of the elements of the descent set.
\end{definition}

\begin{definition}
	Let $D \in \PF(n)$. We define its \emph{parking word} $p(D)$ as follows.

	Let $C_1$ be the set containing the labels appearing in the first column of $D$, and let $p_1(D) \coloneqq \max C_1$. At step $i$, let $C_i$ be the set obtained from $C_{i-1}$ by removing $p_{i-1}(D)$ and adding all the labels in the $i\th$ column of the $D$; let \[ p_i(D) \coloneqq \max \, \{x \in C_i \mid x \leq p_{i-1}(D) \} \] if this last set is non-empty, and $p_i(D) \coloneqq \max \, C_i$ otherwise. We finally define the parking word of $D$ as $p(D) \coloneqq p_1(D) \cdots p_n(D)$.
\end{definition}

\begin{definition}
	We define the statistic \emph{pmaj} on $D \in \PF(n)$ as \[ \pmaj(D) \coloneqq \mathsf{maj}(p_n(D) \cdots p_1(D)). \]
\end{definition}

\begin{example}
	\label{ex:parking_word}
	For example, the parking word of the parking function $D$ in \Cref{fig:LDP_config} is $\overline{1\!0} 9 7 6 5 3 2 \overline{1\!1} 8 4 1 \overline{1\!2}$\footnote{\label{note_bars}We put a bar on the two-digit numbers not to confuse them.}. In fact, we have $C_1 = \{3,4,5,6,9,\overline{1\!0}\}$, $C_2 = \{3,4,5,6,9,7,\overline{1\!1}\}$, $C_3 = \{3,4,5,6,7,\overline{1\!1}\}$, and so on. The descent set of the reverse is $\{1,5\}$, so $\pmaj(D) = 6$. 
\end{example}

\begin{definition}
	For $D \in \PF(n)$ we set $l_i(D)$ to be the label of the $i\th$ vertical step. Then the \emph{pmaj reading word} of $D$ is the sequence $l_1(D) \cdots l_n(D)$, i.e. the sequence of the labels read bottom to top.
\end{definition}

For example, the labelled Dyck path in \Cref{fig:LDP_config} has pmaj reading word $\textcolor{red}{3} \textcolor{blue}{5} \textcolor{olive}{6} 9 \overline{1\!0} \textcolor{olive}{7} \overline{1\!1} \textcolor{red}{2} \textcolor{olive}{8} \textcolor{blue}{4} \textcolor{red}{1} \overline{1\!2}$.

Given two compositions $\mu = (\mu_1,\mu_2,\dots)$ and $\nu = (\nu_1,\nu_2,\dots)$ with $\lvert \mu \rvert + \lvert \nu \rvert=n$, let $K_{\mu_1} = \{n,n-1,\dots, n-\mu_1+1\}$, $K_{\mu_2} = \{n-\mu_1, n-\mu_1-1,\dots,n-\mu_1-\mu_2+1\}$, and so on, and let $I_{\nu_1} = \{1,2,\dots, \nu_1\}$, $I_{\nu_2} = \{\nu_1+1,\nu_1+2,\dots, \nu_1+\nu_2\}$, and so on. Notice that the sets $K_{\mu_1}, K_{\mu_2}, \dots$, $I_{\nu_1}, I_{\nu_2}, \dots$ form a partition of $[n]$.

Now let $\uparrow\!\! K_{\mu_i}$ be the word consisting of the elements of $K_{\mu_i}$ in increasing order: for example $\uparrow\!\! K_{\mu_1}=(n-\mu_1+1) (n-\mu_1+2) \cdots (n-1) n$.
Similarly, let $\downarrow\!\! I_{\nu_j}$ be the word consisting of the elements of $I_{\nu_j}$ in decreasing order: for example $\downarrow\!\! I_{\nu_1}=\nu_1 (\nu_1-1)\cdots 21$.

Consider the shuffle \[ \mathsf{W}(\mu;\nu) \coloneqq \uparrow\!\! K_{\mu_1} \shuffle \uparrow\!\! K_{\mu_2} \shuffle \cdots \shuffle \uparrow\!\! K_{\mu_{\ell(\mu)}} \shuffle \downarrow\!\! I_{\nu_1} \shuffle \downarrow\!\! I_{\nu_1} \shuffle \cdots \shuffle \downarrow\!\! I_{\nu_{\ell(\nu)}}, \]
which we can think of as a set of permutations in $\mathfrak{S}_n$ in one-line notation. Let $\PF(\mu;\nu)$ be the set of parking functions whose pmaj reading word is in $\mathsf{W}(\mu;\nu)$.

For example\footref{note_bars}, $\mathsf{W}((4,3);(3,2)) = 9 \overline{1\!0}\,  \overline{1\!1}\, \overline{1\!2} \shuffle \textcolor{olive}{678} \shuffle \textcolor{blue}{54} \shuffle \textcolor{red}{321}$, and the pmaj reading word of the parking function $D$ in \Cref{fig:LDP_config} belongs to it, so that $D \in \PF((4,3);(3,2))$.

We can now state the shuffle theorem in the form that is suitable for our purposes: this is a combination of the main results in \cite{CarlssonMellitShuffle} and \cite{LoehrRemmel_pmaj} together with the \emph{superization}: see \cite[Chapter~6]{Haglund-Book-2008}.
\begin{theorem}
	\label{thm:shuffle_pmaj}
	For every pair of compositions $\mu$ and $\nu$ with $\lvert \mu \rvert + \lvert \nu \rvert=n$ we have
	\[\langle \nabla e_n, e_\mu h_\nu \rangle = \sum_{D \in \PF(\mu;\nu)} q^{\area(D)} t^{\pmaj(D)}.\]
\end{theorem}

\section{The clique-independent graphs \texorpdfstring{$\widehat{G}_{\mu,\nu}$}{Gmunu}}
\label{sec:Gmunu}

\begin{definition}
	Let $\mu, \nu$ be two compositions (i.e.\ tuples of positive integers). Set $n=\lvert \mu \rvert + \lvert \nu \rvert$. We define a graph $G_{\mu,\nu}$ with set of vertices $[n]$ consisting of the following \emph{components}\footnote{Notice that the notation is consistent with the one used in \Cref{sec:shufflethm}.}:
	\begin{itemize}
		\item $\ell(\mu)$ \emph{clique components}, i.e.\ complete graphs, $K_{\mu_1}, K_{\mu_2}, \dots$, on $\mu_1, \mu_k, \dots$ vertices respectively. The vertices of $K_{\mu_1}$ are $n,n-1,\dots,n-\mu_1+1$; the vertices of $K_{\mu_2}$ are $n-\mu_1,n-\mu_1-1,\dots,n-\mu_1-\mu_2+1$; and so on.
		\item $\ell(\nu)$ \emph{independent components}, i.e.\ graphs without edges, $I_{\nu_1}, I_{\nu_2}, \dots$, on $\nu_1, \nu_2, \dots$ vertices respectively; the vertices of $I_{\nu_1}$ are $1,2,\dots,\nu_1$; the vertices of $I_{\nu_2}$ are $\nu_1+1,\nu_1+2,\dots,\nu_1+\nu_2$; and so on.
	\end{itemize}
	Finally, two vertices in distinct components are always connected by an edge.
\end{definition}

\begin{example}
	If $\mu = \varnothing$, then $G_{\varnothing,\nu}$ is the complete multipartite graph $K_{\nu_1,\nu_2,\dots}$. If $\nu = \varnothing$, then $G_{\mu,\varnothing}$ is isomorphic to the complete graph $K_{\lvert \mu \rvert}$; however, for our purposes we will distinguish between $G_{(\lvert \mu \rvert),\varnothing}$ and $G_{(\mu_1,\mu_2,\dots),\varnothing}$, as we will consider the action of different groups of automorphisms, which will lead to different sorted configurations.
\end{example}

Given one of our labelled graphs $G_{\mu,\nu}$, we define the graph $\widehat{G}_{\mu,\nu}$ simply as $G_{\mu,\nu}$ to which we add a vertex $0$, and we connect it with every other vertex. We will refer to the \emph{clique} and the \emph{independent components} of $\widehat{G}_{\mu,\nu}$ as the corresponding components of $G_{\mu,\nu}$ seen as subgraphs of $\widehat{G}_{\mu,\nu}$. We will consider the sandpile model on $\widehat{G}_{\mu,\nu}$, where $0$ is the sink. \Cref{fig:graph-munu} is an illustration of the graph $\widehat{G}_{(4,\textcolor{olive}{3}),(\textcolor{red}{3},\textcolor{blue}{2})}$.

\begin{figure}
	\centering
	\def\radius{4.8}
	\def\mythick{1.4pt}
	\begin{tikzpicture}[every node/.style={circle,thick,draw,inner sep=1.6pt},scale=0.8]
		\coordinate (0) at (0,0) {};
		\foreach \i in {1,...,12}{
				\coordinate (\i) at (60 + \i*30:\radius) {};
			}
		\def\col{gray!40}
		\foreach \i in {1,...,12}{
				\foreach \j in {\i,...,12}{
						\ifthenelse{\(\i<4 \AND \j<4\) \OR \(\i = 4 \AND \j = 5\)}{\def\col{white}}{}
						\draw[\col] (\i) -- (\j);
					}}
		\foreach \i in {6,...,8}{
				\foreach \j in {\i,...,8}{
						\draw[line width = \mythick,olive] (\i) -- (\j);
					}
			}
		\foreach \i in {9,...,12}{
				\foreach \j in {\i,...,12}{
						\draw[line width = \mythick] (\i) -- (\j);
					}
			}
		\def\col{red}
		\foreach \i in {1,...,12}{
				\draw[gray] (\i)--(0);
				\ifnum\i>3\relax \def\col{blue}\fi
				\ifnum\i>5\relax \def\col{olive}\fi
				\ifnum\i>8\relax \def\col{black}\fi
				\node[\col, fill = white] at (\i) {$\i$};
			}
		\node[gray, fill = white] at (0) {$0$};
	\end{tikzpicture}
	\caption{The graph $\widehat{G}_{(4, \textcolor{olive}{3}), (\textcolor{red}{3}, \textcolor{blue}{2})}$. The vertices $\textcolor{red}{1}$, $\textcolor{red}{2}$ and $\textcolor{red}{3}$ have degree $10$, the vertices $\textcolor{blue}{4}$ and $\textcolor{blue}{5}$ have degree $11$, and all the other vertices have degree $12$.}
	\label{fig:graph-munu}
\end{figure}
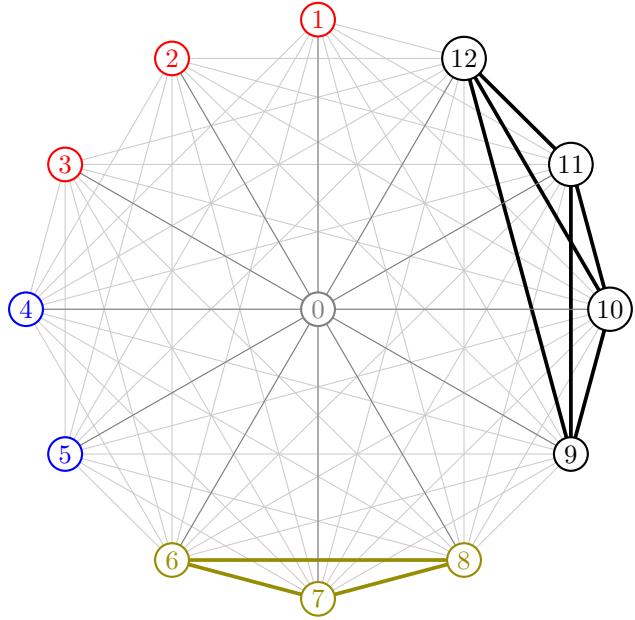

\section{Basics of the sandpile model}
\label{sec:basics-sandpile}

In the present work by a \emph{graph} we will always mean a simple graph, i.e.\ a graph with no loops and no multiple edges.

\begin{definition}
	Let $G$ be a finite, undirected graph on the vertex set $\{0,1,\dots,n\}$.

	A \emph{configuration} of the \emph{sandpile (model)} on $G$ is a map $\kappa \colon \{0\} \cup [n] \to \Z$ that assigns a (integer) number of ``grains of sand'' to each nonzero vertex of $G$.

	If $0 \leq \kappa(v)$ we say that $v$ is \emph{non-negative}. If $\kappa(v) \leq \deg(v)$, we say that $v$ is \emph{stable}, and otherwise it is \emph{unstable}. Any vertex can \emph{topple} (or \emph{fire}), and ``donate a single grain'' to each of its neighbors: the result is a new configuration $\kappa'$ in which $\kappa'(v) = \kappa(v) - \deg(v)$ and for any $w \neq v$
	\[\kappa'(w) =
		\begin{cases}
			\kappa(w) + 1, & \text{if }(v,w) \text{ is an edge} \\
			\kappa(w),     & \text{otherwise.}
		\end{cases}
	\]
	For any $v \in \{0,\dots,n\}$ we write $\phi_v$ for the \emph{toppling operator} at vertex $v$. That is $\phi_v(\kappa)$ is a new configuration obtained from $\kappa$ by toppling the vertex $v$.

	The vertex $0$ is special in this model, and we call it the \emph{sink}, while we call all the others \emph{non-sink} vertices. We say that a configuration $\kappa$ is \emph{non-negative} if all of its non-sink vertices are non-negative, \emph{stable} if all of its non-sink vertices are stable, and \emph{unstable} if at least one of its non-sink vertices is unstable.
\end{definition}

\begin{remark}
	Notice that the notion of stable configuration has no dependency on the value on the sink. Therefore, as it is customary, we will ignore the value of a configuration on the sink, and consider the configurations as restricted to the non-sink vertices. Moreover, we will identify every configuration $\kappa$ with the word $\kappa(n) \kappa(n-1) \cdots \kappa(2) \kappa(1)$.
\end{remark}

\begin{example}
	\label{ex:kappa_config}
	Consider the graph $\widehat{G}_{(4,\textcolor{olive}{3}),(\textcolor{red}{3},\textcolor{blue}{2})}$ (see \Cref{fig:graph-munu}), whose vertices are $\{0\} \cup [12]$, and let $0$ be the sink. The configuration $\kappa$ given by\footref{note_bars}  $3 \overline{1\!0}\, \overline{1\!1}\, \overline{1\!1} \textcolor{olive}{8 \overline{1\!0}\, \overline{1\!1}}\, \textcolor{blue}{\overline{1\!0} 4} \textcolor{red}{9 7 3}$ is a stable configuration. We provide below some examples of such topplings:
	\begin{align*}
		\phi_0(\kappa)
		 & = 4 \overline{1\!1}\, \overline{1\!2}\, \overline{1\!2} \textcolor{olive}{9 \overline{1\!1}\, \overline{1\!2}}\, \textcolor{blue}{\overline{1\!1} 5} \textcolor{red}{\overline{1\!0} 8 4},    \\
		(\phi_{\overline{1\!0}}\!\circ\!\phi_0)(\kappa)
		 & = 5 \overline{1\!2} 0 \overline{1\!3}\, \textcolor{olive}{\overline{1\!0}\, \overline{1\!2}\, \overline{1\! 3}}\, \textcolor{blue}{\overline{1\!2} 6} \textcolor{red}{\overline{1\!1} 9 5},   \\
		(\phi_{9}\!\circ\!\phi_{\overline{1\!0}}\!\circ\!\phi_0)(\kappa)
		 & = 6 \overline{1\!3}\, 1 1 \textcolor{olive}{\overline{1\!1}\, \overline{1\!3}\, \overline{1\! 4}}\, \textcolor{blue}{\overline{1\!3} 7} \textcolor{red}{\overline{1\!2}\, \overline{1\!0} 6}, \\
		(\phi_{7}\!\circ\! \phi_{9}\!\circ\!\phi_{\overline{1\!0}}\!\circ\!\phi_0)(\kappa)
		 & = 7 \overline{1\!4}\, 2 2 \textcolor{olive}{\overline{1\!2} 1 \overline{1\! 5}}\, \textcolor{blue}{\overline{1\! 4} 8} \textcolor{red}{\overline{1\! 3}\, \overline{1\!1} 7},                 \\
		(\phi_{6}\!\circ\!\phi_{7}\!\circ\! \phi_{9}\!\circ\!\phi_{\overline{1\!0}}\!\circ\!\phi_0)(\kappa)
		 & = 8 \overline{1\!5}\, 3 3 \textcolor{olive}{\overline{1\!3} 2 3} \textcolor{blue}{\overline{1\!5} 9} \textcolor{red}{\overline{1\!4}\, \overline{1\!2} 8},                                    \\
		(\phi_{5}\!\circ\!\phi_{6}\!\circ\!\phi_{7}\!\circ\! \phi_{9}\!\circ\!\phi_{\overline{1\!0}}\!\circ\!\phi_0)(\kappa)
		 & = 9 \overline{1\!6}\, 4 4 \textcolor{olive}{\overline{1\!4} 3 4} \textcolor{blue}{4 9} \textcolor{red}{\overline{1\!5}\, \overline{1\!3} 9},                                                  \\
		(\phi_{3}\!\circ\!\phi_{5}\!\circ\!\phi_{6}\!\circ\!\phi_{7}\!\circ\!\phi_{9}\!\circ\!\phi_{\overline{1\!0}}\!\circ\!\phi_0)(\kappa)
		 & = \overline{1\!0}\, \overline{1\!7}\, 5 5 \textcolor{olive}{\overline{1\!5} 4 5} \textcolor{blue}{5 \overline{1\!0}} \textcolor{red}{5 \overline{1\!3} 9},                                    \\
		(\phi_{2}\!\circ\!\phi_{3}\!\circ\!\phi_{5}\!\circ\!\phi_{6}\!\circ\!\phi_{7}\!\circ\!\phi_{9}\!\circ\!\phi_{\overline{1\!0}}\!\circ\!\phi_0)(\kappa)
		 & = \overline{1\!1}\, \overline{1\!8}\, 6 6 \textcolor{olive}{\overline{1\!6} 5 6} \textcolor{blue}{6 \overline{1\!1}} \textcolor{red}{5 3 9}.
	\end{align*}
\end{example}

\begin{definition}
	\label{def:recurrent}
	Let $\kappa$ be a stable configuration, and consider the configuration $\phi_0(\kappa)$. We say that $\kappa$ is \emph{recurrent}\footnote{In the literature ``recurrent'' is sometimes used in a broader sense than in this paper. Configurations that are recurrent in our sense are called \emph{critical} in these settings. The notion of a recurrent configuration comes from a Markov chain on sandpile configurations not detailed here. The algorithmic characterisation that we used here as definition is called the \emph{Dhar criterion}.} if from $\phi_0(\kappa)$ there is an order of all the non-sink vertices such that by toppling the vertices in that precise order we always stay non-negative. Of course at the end of this sequence of topplings we will be back to $\kappa$, since each edge will be crossed by a single grain in each direction. More precisely, a stable configuration $\kappa$ is recurrent if there is a permutation $\sigma=\sigma_1\sigma_2\cdots \sigma_n\in \mathfrak{S}_n$ such that
	\[\phi_0(\kappa),(\phi_{\sigma(1)}\circ \phi_0)(\kappa),(\phi_{\sigma(2)}\circ\phi_{\sigma(1)}\circ \phi_0)(\kappa),\dots,(\phi_{\sigma(n)}\circ \cdots \circ \phi_{\sigma(1)}\circ \phi_0)(\kappa) = \kappa\]
	are all non-negative configurations. In this case, $\sigma$ is the \emph{toppling word} of this sequence of topplings, and we say that this sequence \emph{verifies the recurrence} of $\kappa$.
\end{definition}

\begin{example} \label{ex:verify_recurrence}
	The configuration $\kappa = 3 \overline{1\!0}\, \overline{1\!1}\, \overline{1\!1} \textcolor{olive}{8 \overline{1\!0}\, \overline{1\!1}}\, \textcolor{blue}{\overline{1\!0} 4} \textcolor{red}{9 7 3}$ is a recurrent configuration for $\widehat{G}_{(4,\textcolor{olive}{3}),(\textcolor{red}{3},\textcolor{blue}{2})}$: indeed it is easy to check that $\sigma = \overline{1\!0} 9 7 6 5 3 2 \overline{1\!1} 8 4 1 \overline{1\!2}$ verifies the recurrence of $\kappa$ (cf.\ \Cref{ex:parking_word}).
\end{example}

\begin{remark}\label{rem:recurrence_definition}
	It is well known (see e.g.\ \cite[Theorem~2.4]{ADDL_two_op}) that the condition for $\kappa$ to be recurrent is equivalent to say that starting from $\phi_0(\kappa)$ there is no proper (possibly empty) subset $A$ of $[n]$ such that toppling all the vertices of $A$ brings $\phi_0(\kappa)$ to a stable configuration.
\end{remark}

\begin{definition}
	Given a recurrent configuration $\kappa$ of $G$, we define its \emph{level} as
	\[\lev(\kappa) \coloneqq - \lvert E_s(G) \rvert + \sum_{i=1}^n \kappa(i)\]
	where $E_s(G)$ is the set of edges of $G$ that are not incident to the sink.
\end{definition}

It is well-known that $\lev(\kappa)\geq 0$, and there exists a recurrent configuration of level $0$ if $G$ is connected \cite{MerinoLopez}.

\begin{remark}
	\label{rmk:level_Gmunu}
	For $\widehat{G}_{\mu,\nu}$ with $\lvert \mu \rvert + \lvert \nu \rvert = n$, (by removing from the complete graph on $\{0\}\cup [n]$ all the edges between vertices in the same independent component) we have
	\[ \lvert E_s(\widehat{G}_{\mu,\nu}) \rvert =\binom{n}{2} - \sum_{i \geq 0} \binom{\nu_i}{2}.\]
\end{remark}

\begin{example}
	\label{ex:level}
	The configuration  $\kappa = 3 \overline{1\!0}\, \overline{1\!1}\, \overline{1\!1} \textcolor{olive}{8 \overline{1\!0}\, \overline{1\!1}}\, \textcolor{blue}{\overline{1\!0} 4} \textcolor{red}{9 7 3}$ for $\widehat{G}_{(4,\textcolor{olive}{3}),(\textcolor{red}{3},\textcolor{blue}{2})}$ has level
	\[ \lev(\kappa) = -\binom{12}{2} + \binom{\textcolor{red}{3}}{2} + \binom{\textcolor{blue}{2}}{2} + 97 = 35.\]
\end{example}

\section{The toppling algorithm and the \texorpdfstring{$\del$}{delay} statistic}
\label{sec:algo}

Consider the sandpile model on a graph $G$ with vertices $\{0\}\cup [n]$, where $0$ is the sink. Let $\kappa$ be a recurrent configuration of $G$. Consider \Cref{alg:toppling}. Before discussing the algorithm, let us look at an example.
\begin{example}
	\label{ex:algo}
	Consider again the configuration $\kappa$ from \Cref{ex:kappa_config}: in that example we actually computed the sequence of topplings given by the first iteration of the \textbf{for} loop of \Cref{alg:toppling} applied to $\kappa$. We compute the second iteration of the \textbf{for} loop:
	\begin{align*}
		(\phi_{\overline{1\!1}}\!\circ\!\phi_{2}\!\circ\!\phi_{3}\!\circ\!\phi_{5}\!\circ\!\phi_{6}\!\circ\!\phi_{7}\!\circ\! \phi_{9}\!\circ\!\phi_{\overline{1\!0}}\!\circ\!\phi_0)(\kappa)
		 & =\overline{1\!2}6 7 7 \textcolor{olive}{\overline{1\! 7} 6 7} \textcolor{blue}{7 \overline{1\!2}} \textcolor{red}{6 4 \overline{1\!0}},   \\
		(\phi_{8}\!\circ\!\phi_{\overline{1\!1}}\!\circ\!\phi_{2}\!\circ\!\phi_{3}\!\circ\!\phi_{5}\!\circ\!\phi_{6}\!\circ\!\phi_{7}\!\circ\! \phi_{9}\!\circ\!\phi_{\overline{1\!0}}\!\circ\!\phi_0)(\kappa)
		 & =\overline{1\! 3}7 8 8 \textcolor{olive}{5 7 8} \textcolor{blue}{8 \overline{1\! 3}} \textcolor{red}{7 5 \overline{1\!1}},                \\
		(\phi_{4}\!\circ\!\phi_{8}\!\circ\!\phi_{\overline{1\!1}}\!\circ\!\phi_{2}\!\circ\!\phi_{3}\!\circ\!\phi_{5}\!\circ\!\phi_{6}\!\circ\!\phi_{7}\!\circ\! \phi_{9}\!\circ\!\phi_{\overline{1\!0}}\!\circ\!\phi_0)(\kappa)
		 & =\overline{1\! 4}8 9 9 \textcolor{olive}{6 8 9} \textcolor{blue}{8 2} \textcolor{red}{8 6 \overline{1\!2}},                               \\
		(\phi_{1}\!\circ\!\phi_{4}\!\circ\!\phi_{8}\!\circ\!\phi_{\overline{1\!1}}\!\circ\!\phi_{2}\!\circ\!\phi_{3}\!\circ\!\phi_{5}\!\circ\!\phi_{6}\!\circ\!\phi_{7}\!\circ\! \phi_{9}\!\circ\!\phi_{\overline{1\!0}}\!\circ\!\phi_0)(\kappa)
		 & =\overline{1\! 5}9 \overline{1\!0}\, \overline{1\!0} \textcolor{olive}{7 9 \overline{1\!0}} \textcolor{blue}{9 3} \textcolor{red}{8 6 2},
	\end{align*}
	and finally the third and last iteration of the \textbf{for} loop:
	\begin{align*}
		(\phi_{\overline{1\!2}}\!\circ\!\phi_{1}\!\circ\!\phi_{4}\!\circ\!\phi_{8}\!\circ\!\phi_{\overline{1\!1}}\!\circ\!\phi_{2}\!\circ\!\phi_{3}\!\circ\!\phi_{5}\!\circ\!\phi_{6}\!\circ\!\phi_{7}\!\circ\! \phi_{9}\!\circ\!\phi_{\overline{1\!0}}\!\circ\!\phi_0)(\kappa) & = 3 \overline{1\!0}\, \overline{1\!1}\, \overline{1\!1} \textcolor{olive}{8 \overline{1\!0}\, \overline{1\!1}}\, \textcolor{blue}{\overline{1\!0} 4} \textcolor{red}{9 7 3} = \kappa.
	\end{align*}
	Hence, the output of \Cref{alg:toppling} applied to $\kappa$ is the word $\overline{1\!0} 9 7 6 5 3 2\overline{1\!1} 8 41\overline{1\!2}$ (cf.\ \Cref{ex:parking_word} and \Cref{ex:verify_recurrence}).
\end{example}

\begin{algorithm}
	\caption{Toppling algorithm}
	\label{alg:toppling}
	\begin{algorithmic}
		\Require A graph $G$ with $n$ non-sink vertices and a recurrent configuration $\kappa$
		\Ensure The word of non-sink vertices in the order they have been toppled
		\State Topple the sink, i.e.\ compute $\phi_0(\kappa)$
		\State Initialize the output word as empty
		\While{there are non-sink vertices that are untoppled}
		\For{$i$ going from $n$ to $1$ (in decreasing order)}
		\If{vertex $i$ is unstable}
		\State Topple vertex $i$ \Comment{possibly creating new unstable vertices}
		\State Append $i$ to the output word
		\EndIf
		\EndFor
		\EndWhile
	\end{algorithmic}
\end{algorithm}

Observe that, by construction, the algorithm terminates: since $\kappa$ is recurrent,  $\phi_0(\kappa)$ is non-negative and at least one of the vertices adjacent to the sink is unstable; then every time we topple we stay non-negative, and since $\kappa$ is recurrent the process must go through all the non-sink vertices (otherwise we found a subset $A$ of non-sink vertices such that after we topple its vertices we are in a stable configuration, cf.\ \Cref{rem:recurrence_definition}).

By construction the algorithm outputs a toppling sequence that verifies the recurrence of $\kappa$.
We can now define our new statistic on recurrent configurations.
\begin{definition}
	Let $\kappa$ be a recurrent configuration of $G$. For every $i \in [n]$, let $r_i(\kappa)$ be the number of \textbf{for} loop iterations in \Cref{alg:toppling} that occurred before the one in which the vertex $i$ is toppled (so if $i$ is toppled in the first iteration, then $r_i(\kappa)=0$). Then we define the \emph{delay} of $\kappa$ as
	\[\del(\kappa) \coloneqq \sum_{i=1}^n r_i(\kappa).\]
\end{definition}

\begin{remark}
	\label{rmk:delay}
	If $\sigma$ is the output of \Cref{alg:toppling} applied to $\kappa$, then, since clearly descents in the permutation $\sigma_n\sigma_{n-1}\cdots \sigma_1$ correspond to \textbf{while} loop iterations in the computation of the delay, we have
	\[\del(\kappa) = \maj(\sigma_n\sigma_{n-1}\cdots \sigma_1).\]
\end{remark}

\begin{example}
	\label{ex:delay}
	For the configuration $\kappa$ of \Cref{ex:kappa_config}, we found in \Cref{ex:algo} that \Cref{alg:toppling} gives $\sigma = \overline{1\!0} 9 7 6 5 3 2 \overline{1\!1} 8 4 1 \overline{1\!2}$, so that $\del(\kappa) = \maj (\overline{1\!2} 1 4 8 \overline{1\!1} 2 3 5 6 7 9 \overline{1\!0}) = 1+5 = 6$. Actually, looking at the computation of the algorithm, we find that the word $r_1(\kappa)r_2(\kappa)\cdots$ in this case is $100100010012$, whose letters add up to $6$ (cf.\ \Cref{ex:parking_word}).
\end{example}

\section{Sorted recurrent configurations of \texorpdfstring{$\widehat{G}_{\mu,\nu}$}{Gmunu}} \label{sec:SortRec}

Let $\mu$ and $\nu$ be two compositions such that $\lvert \mu \rvert + \lvert \nu \rvert=n$, and consider the Young subgroup $\mathfrak{S}_\mu \times \mathfrak{S}_\nu$ of the symmetric group $\mathfrak{S}_n$ consisting of the permutations that preserve the components of $G_{\mu,\nu}$. We want to consider configurations modulo the natural action of $\mathfrak{S}_\mu \times \mathfrak{S}_\nu$ on the set of configurations. More precisely, a \emph{sorted configuration}\footnote{The relation with the general definition of \emph{sorted configuration} given in the introduction is simply that we are picking a specific convenient element in each orbit.} of the sandpile on $\widehat{G}_{\mu,\nu}$ is a configuration $\kappa$ that is weakly decreasing inside each clique component of $\widehat{G}_{\mu,\nu}$ and weakly increasing inside each independent component of $\widehat{G}_{\mu,\nu}$: if $i,j \in K_{\mu_r}$ and $i<j$, then $\kappa(i) \geq \kappa(j)$; if $i,j \in I_{\nu_s}$ and $i<j$, then $\kappa(i) \leq \kappa(j)$.

\begin{example}
	\label{ex:sorted_recurrent_config}
	The configuration \[ \kappa = 3 \overline{1\!0}\, \overline{1\!1}\, \overline{1\!1} \textcolor{olive}{8 \overline{1\!0}\, \overline{1\!1}}\, \textcolor{blue}{\overline{1\!0} 4} \textcolor{red}{9 7 3} \] is a sorted recurrent configuration of $\widehat{G}_{(4,\textcolor{olive}{3}),(\textcolor{red}{3},\textcolor{blue}{2})}$ (recall that in our notation $\kappa = \kappa(n)\kappa(n-1)\cdots \kappa(1)$).
\end{example}

Let $\kappa$ be a sorted recurrent configuration of $\widehat{G}_{\mu,\nu}$. Let $\sigma \in \mathfrak{S}_n$ be the toppling word produced by \Cref{alg:toppling} applied to $\kappa$. For every vertex $v$ in a clique component $K_{\mu_r}$ we set\footnote{The idea of going from $\kappa$ to $\widetilde{\kappa}$ is to ``embed'' the configurations of $\widehat{G}_{\mu,\nu}$ into the configurations of $\widehat{G}_{\mu \sqcup \nu,\varnothing}$ (which is defined below).
	It corresponds to add back the removed edges in the independent components, without modifying the level of the configurations, nor the output of Algorithm~\ref{alg:toppling} (cf.\ Lemma~\ref{lem:indep_reduction}).} \[ \widetilde{\kappa}(v) \coloneqq \kappa(v). \]

For every independent component $I_{\nu_s}$ of $G_{\mu,\nu}$, we order its vertices in decreasing order, and if $v^{(s)}_j$ is the $j\th$ vertex of $I_{\nu_s}$, we set \[ \widetilde{\kappa}(v^{(s)}_j) \coloneqq \kappa(v^{(s)}_j) + \nu_s - j.\]

\begin{remark}
	\label{rmk:kappatilde_increasing}
	Since $\kappa$ is sorted, it is weakly increasing on the vertices of $I_{\nu_s}$, i.e., using the notation above, $\kappa(v^{(s)}_j) \geq \kappa(v^{(s)}_{j+1})$ for every $j \in [\nu_s-1]$. Therefore $\widetilde{\kappa}$ is strictly increasing on the vertices of $I_{\nu_s}$, i.e.\ \[\widetilde{\kappa}(v^{(s)}_j) = \kappa(v^{(s)}_j) + \nu_s - j \geq \kappa(v^{(s)}_{j+1})+ \nu_s - j > \kappa(v^{(s)}_{j+1}) + \nu_s-(j+1) = \widetilde{\kappa}(v^{(s)}_{j+1}) \]
	for every $j \in [\nu_s-1]$.
\end{remark}

Given two compositions $\mu$ and $\nu$, let $\widehat{G}_{\mu \sqcup \nu, \varnothing}$ be the labelled graph obtained from $\widehat{G}_{\mu,\nu}$ by turning its independent components into clique components, i.e.\ by adding all the missing edges; but notice that for this graph we keep the same notation for the components as the one for $\widehat{G}_{\mu,\nu}$, and we call a stable configuration \emph{sorted} if it is weakly decreasing in the original clique components $K_{\mu_s}$ and weakly increasing in the new clique components $I_{\nu_r}$, i.e.\ the ones coming from independent components of $\widehat{G}_{\mu, \nu}$.

\begin{lemma}
	\label{lem:indep_reduction}
	Given $\mu$ and $\nu$ two compositions, if $\kappa$ is a sorted stable configuration for $\widehat{G}_{\mu,\nu}$, then $\widetilde{\kappa}$ is sorted stable for $\widehat{G}_{\mu \sqcup \nu, \varnothing}$. Moreover, if $\kappa$ is recurrent for $\widehat{G}_{\mu,\nu}$, then $\widetilde{\kappa}$ is recurrent for $\widehat{G}_{\mu \sqcup \nu, \varnothing}$, and in this case if $\sigma \in \mathfrak{S}_n$ is the toppling word produced by \Cref{alg:toppling} applied to $\kappa$ on $\widehat{G}_{\mu,\nu}$, then $\sigma$ is also the toppling word produced by \Cref{alg:toppling} applied to $\widetilde{\kappa}$ on $\widehat{G}_{\mu \sqcup \nu,\varnothing}$.
\end{lemma}

\begin{proof}
	The stability of $\widetilde{\kappa}$ for $\widehat{G}_{\mu \sqcup \nu, \varnothing}$ is clear, since, going from $\widehat{G}_{\mu,\nu}$ to $\widehat{G}_{\mu \sqcup \nu, \varnothing}$, all that happens is that for every $r$ the degree of the vertices in $I_{\nu_r}$ goes up by $\nu_r-1$. If $\kappa$ is sorted for $\widehat{G}_{\mu,\nu}$, then by definition $\widetilde{\kappa}$ is sorted for $\widehat{G}_{\mu \sqcup \nu, \varnothing}$.

	About $\sigma$: let $\kappa$ be sorted recurrent for $\widehat{G}_{\mu,\nu}$, and suppose that for the first $N$ visited\footnote{\label{ftn:visits}\Cref{alg:toppling} \emph{visits} all the non-sink vertices in decreasing order during each iteration of the \textbf{for} loop: these are the visits we are talking about.} vertices \Cref{alg:toppling} did exactly the same topplings while running on $\widehat{G}_{\mu,\nu}$ for $\kappa$ and while running on $\widehat{G}_{\mu \sqcup \nu, \varnothing}$ for $\widetilde{\kappa}$. So we reached the configurations $u$ and $\widetilde{u}$ respectively, and now it comes the $(N+1)^{\text{st}}$ vertex to visit, say $v$, which belongs to a given component.
	If we topple a vertex in a different component, the effect on the values on the vertices of our component is the same in $\widehat{G}_{\mu,\nu}$ and in $\widehat{G}_{\mu\sqcup \nu,\varnothing}$. If our given component is a clique component, then toppling one of its vertices is going to have the same effect in $\widehat{G}_{\mu,\nu}$ and in $\widehat{G}_{\mu \sqcup \nu,\varnothing}$ on the values at the vertices of our component.

	So if the component of $v$ is a clique component, then it is clear that $u(v)= \widetilde{u}(v)$ and \Cref{alg:toppling} is going to topple $v$ while running on $\widehat{G}_{\mu,\nu}$ for $\kappa$ if and only if it topples $v$ while running on $\widehat{G}_{\mu\sqcup \nu,\varnothing}$ for $\widetilde{\kappa}$.
	Suppose now that $v$ belongs to an independent component $I_{\nu_r}$ of $\widehat{G}_{\mu,\nu}$. Let $v_1,v_2,\dots,v_{\nu_r}$ be the vertices of $I_{\nu_r}$ in decreasing order, so that for every $i\in [\nu_r]$ we have
	\[\deg_{\widehat{G}_{\mu,\nu}}(v_i)-\kappa(v_i)= \deg_{\widehat{G}_{\mu \sqcup \nu, \varnothing}}(v_i)-\widetilde{\kappa}(v_i)-i+1.\]
	Let $m \in [\nu_r]$ be such that $v = v_m$, and let $b$ be the number of $j \in [\nu_r]$ such that $v_j$ has been already toppled in the first $N$ visits\footref{ftn:visits} of \Cref{alg:toppling}. We have two cases.

	\underline{Case 1}: $v_m$ has been toppled in the first $N$ visits of the algorithm. In this case
	\begin{align*}
		\tilde{u}(v_m) & =\tilde\kappa(v_m) + (N-1) -\deg_{\widehat{G}_{\mu\sqcup\nu,\varnothing}}(v_m) = \tilde\kappa(v_m) + N -1 - n \\
		u(v_m)         & =\kappa(v_m) + (N-b) - \deg_{\widehat G_{\mu,\nu}}(v_m) = \kappa(v_m) + N - b - (n-\nu_r +1)
	\end{align*}
	and so 
	\[\tilde u(v_m)-u(v_m)-b+\nu_r = \tilde{\kappa}(v_m) - \kappa(v_m) = \nu_r-m,\]
	hence
	\begin{align*}
		\deg_{\widehat{G}_{\mu,\nu}}(v_m)-u(v_m) & =\deg_{\widehat{G}_{\mu,\nu}}(v_m)-\widetilde{u}(v_m)+b-m                             \\
		                                         & =\deg_{\widehat{G}_{\mu \sqcup \nu, \varnothing}}(v_m)-\widetilde{u}(v_m)+b-m-\nu_r+1 \\
		  & \leq \deg_{\widehat{G}_{\mu \sqcup \nu, \varnothing}}(v_m)-\widetilde{u}(v_m),
	\end{align*}
where the last inequality fomes from the fact that $b\leq \nu_r$ and $1\leq m$.
	Also, clearly in this case $\deg_{\widehat{G}_{\mu,\nu}} (v_m) - u(v_m) > 0$, so that $\deg_{\widehat{G}_{\mu \sqcup \nu, \varnothing}}(v_m)-\widetilde{u}(v_m)> 0$, and in both situations $v_m$ does not get toppled in the $(N+1)^{\text{st}}$ visit of the algorithm.

	\underline{Case 2}: $v_m$ has not been toppled in the first $N$ visits of the algorithm. In this case
	\[\widetilde{u}(v_m)-u(v_m)-b=\widetilde{\kappa}(v_m)-\kappa(v_m) = \nu_r-m,\]
	hence
	\[ \deg_{\widehat{G}_{\mu,\nu}}(v_m)-u(v_m) = \deg_{\widehat{G}_{\mu \sqcup \nu, \varnothing}}(v_m)-\widetilde{u}(v_m)+b-m+1. \]
	We observe that if $t>1$ is such that $v_t$ has been toppled during the first $N$ visits of the algorithm running on $\widehat{G}_{\mu,\nu}$, then $v_{t-1}$ has also been toppled. Indeed, let $\eta$ be the configuration right before the toppling of $v_t$, and suppose to the contrary that $v_{t-1}$ has not yet been toppled either. This means that $v_{t-1}$ and $v_{t}$ have received the same number of grains and neither has donated any; hence
	\[\eta(v_{t-1})-\eta(v_{t}) = \kappa(v_{t-1})-\kappa(v_{t}) \geq 0, \]
	and so \[\deg_{\widehat{G}_{\mu,\nu}}(v_{t-1})-\eta(v_{t-1})=\deg_{\widehat{G}_{\mu,\nu}}(v_{t})-\eta(v_{t-1})\leq \deg_{\widehat{G}_{\mu,\nu}}(v_{t})-\eta(v_{t})\leq 0,\]
	so $v_{t-1}$ is unstable and should have been toppled during the previous visit.

	This observation implies that the set of $j\in [\nu_r]$ such that $v_j$ has been toppled during the first $N$ visits of the algorithm is precisely $[b]$. Now
	\[\deg_{\widehat{G}_{\mu,\nu}}(v_m)-u(v_m)=\deg_{\widehat{G}_{\mu \sqcup \nu, \varnothing}}(v_m)-\widetilde{u}(v_m)+b-m+1\leq \deg_{\widehat{G}_{\mu \sqcup \nu, \varnothing}}(v_m)-\widetilde{u}(v_m)\]
	implies that if $v_m$ is toppled at the $(N+1)^{\text{st}}$ visit of the algorithm running on $\widehat{G}_{\mu\sqcup \nu,\varnothing}$, then it will certainly be also toppled at the $(N+1)^{\text{st}}$ visit of the algorithm running on $\widehat{G}_{\mu, \nu}$. On the other hand, if $v_m$ is toppled at the $(N+1)^{\text{st}}$ visit of the algorithm running on $\widehat{G}_{\mu, \nu}$, then the above observation implies that necessarily $b=m-1$, hence
	\[\deg_{\widehat{G}_{\mu,\nu}}(v_m)-u(v_m) =\deg_{\widehat{G}_{\mu\sqcup \nu,\varnothing}}(v_m)-\widetilde{u}(v_m)+b-m+1 = \deg_{\widehat{G}_{\mu\sqcup \nu,\varnothing}}(v_m)-\widetilde{u}(v_m),\]
	so that  $v_m$ is toppled also at the $(N+1)^{\text{st}}$ visit of the algorithm running on $\widehat{G}_{\mu\sqcup \nu,\varnothing}$.

	This shows that \Cref{alg:toppling} running for $\kappa$ on $\widehat{G}_{\mu, \nu}$ produces exactly the same output when running for $\widetilde{\kappa}$ on $\widehat{G}_{\mu \sqcup \nu, \varnothing}$, proving that $\widetilde{\kappa}$ is also recurrent, and that its toppling word is exactly the same as the one of $\kappa$, as claimed.
\end{proof}

Given two compositions $\mu$ and $\nu$ with $ \lvert \mu \rvert + \lvert \nu \rvert = n$, a sorted stable configuration $\kappa$ of $\widehat{G}_{\mu,\nu}$, and a permutation $\sigma\in \mathfrak{S}_n$, for every $i\in [n]$ we set\footnote{There is no relation with the $u$ in the proof of \Cref{lem:indep_reduction}.}
\[u_{\sigma^{-1}(i)} \coloneqq \sigma^{-1}(i)+\widetilde{\kappa}(i)-n.\]
An interpretation of $u_{\sigma^{-1}(i)}$ in terms of the sandpile model will be provided in the course of the proof of Proposition~\ref{prop:recurrent_inequalities}. Let us first look at an example.

\begin{example}
	\label{ex:u}
	For the configuration $\kappa = 3 \overline{1\!0}\, \overline{1\!1}\, \overline{1\! 1} \textcolor{olive}{8 \overline{1\!0}\, \overline{1\!1}}\, \textcolor{blue}{\overline{1\!0} 4} \textcolor{red}{9 7 3}$ in \Cref{ex:sorted_recurrent_config}, we found in \Cref{ex:algo} that \Cref{alg:toppling} gives $\sigma=\overline{1\!0} 9 7 6 5 3 2\overline{1\!1} 8 4 1 \overline{1\!2}$. Hence, $\widetilde{\kappa}=3\overline{1\!0}\, \overline{1\!1}\, \overline{1\!1} \textcolor{olive}{8 \overline{1\!1}\, \overline{1\!1}}\, \textcolor{blue}{\overline{1\!1} 4} \textcolor{red}{\overline{1\!1} 8 3}$, and the word $u \coloneqq u_1 u_2 \cdots$ is $011345365223$.
\end{example}

Given a permutation $\sigma = \sigma(1)\sigma(2)\cdots \sigma(n)$, we add $\sigma(0) \coloneqq 0$ in front of it, and we define its \emph{runs} as its maximal consecutive decreasing substrings. Now for every $i \in [n]$, we define $w_{\sigma^{-1}(i)} = w_{\sigma^{-1}(i)}(\sigma)$ as
\begin{align*}
	w_{\sigma^{-1}(i)}(\sigma) & \coloneqq \#\{\text{numbers in the same run of $i$ and larger than it}\}                                   \\
	                           & +\#\{\text{smaller numbers in the run immediately to the left of the one containing $i$}\}.
\end{align*}
\begin{example}\label{ex:w}
	The runs of $\sigma = \overline{1\!0} 9 7 6 5 3 2 \overline{1\!1} 8 4 1 \overline{1\!2}$ are separated by bars: $0 | \overline{1\!0} 9 7 6 5 3 2 | \overline{1\!1} 8 4 1 | \overline{1\!2}$, so that the word $w = w_1(\sigma) w_2(\sigma) \cdots $ is $123456776434$.
\end{example}

The following propositions characterize the sorted recurrent configurations of $\widehat{G}_{\mu,\nu}$.
\begin{proposition}
	\label{prop:recurrent_inequalities}
	Let $\kappa$ be a sorted recurrent configuration of $\widehat{G}_{\mu,\nu}$. Let $\sigma \in \mathfrak{S}_n$ be the toppling word produced by \Cref{alg:toppling} applied to $\kappa$. Then for every $i \in [n]$ \[0 \leq u_{\sigma^{-1}(i)} < w_{\sigma^{-1}(i)}.\]
\end{proposition}

\begin{proof}
	By \Cref{lem:indep_reduction}, $\sigma$ is also the toppling word computed by the algorithm for $\widetilde{\kappa}$ on $\widehat{G}_{\mu \sqcup \nu, \varnothing}$. So let us focus on the toppling process of $\tilde \kappa$ in this complete graph.

	The first inequality comes from the fact that $\sigma^{-1}(i) + \tilde\kappa(i)$ is the value of the configuration on $i$, at the moment that we must topple $i$. Indeed,
	\begin{align*}
		\sigma^{-1}(i) & = \text{position of $i$ in the toppling algorithm}               \\
		               & = \#\{\text{vertices (including the sink) toppled before $i$}\}.
	\end{align*}
	Since $\deg_{\widehat{G}_{\mu\sqcup \nu,\varnothing}}(i)=n$, and we are toppling $i$, we must have $\sigma^{-1}(i)+\tilde{\kappa}(i) \geq n$ and so $u_{\sigma^{-1}(i)}\geq 0$. It also follows that $u_{\sigma^{-1}(i)}$ is the number of grains remaining at $i$ right after it topples.

	Now for the second inequality. Notice that the decreasing runs of $\sigma$ are exactly the vertices that are toppled in each iteration of the \textbf{for} loop of \Cref{alg:toppling}. It follows that $w_{\sigma^{-1}(i)}$ is the number of vertices toppled between the time we last checked $i$ and the time we actually topple $i$ in the algorithm.
	This number cannot be less than or equal to $u_{\sigma^{-1}(i)}$, or else $i$ would have been able to topple the last time it was checked. Indeed, $u_{\sigma^{-1}(i)}$ is the number of grains on $i$ \emph{after} $i$ is toppled; if $w_{\sigma^{-1}(i)} \leq u_{\sigma^{-1}(i)}$, then $i$ would already have been unstable when it was last visited.
\end{proof}

We have a converse of the previous proposition.

\begin{proposition}
	Let $\kappa$ be a sorted stable configuration of $\widehat{G}_{\mu,\nu}$, and let $\sigma \in \mathfrak{S}_n$ be such that for every $i\in [n]$
	\[0 \leq u_{\sigma^{-1}(i)} < w_{\sigma^{-1}(i)}.\]
	Then $\kappa$ is recurrent and $\sigma$ is the toppling word given by \Cref{alg:toppling} applied to $\kappa$.
\end{proposition}

\begin{proof}
	To show that $\kappa$ is recurrent it is enough to show that, after we topple the sink, if we topple the non-sink vertices in the order given by $\sigma$, then at any time we have a non-negative configuration. Using the interpretation of $u_{\sigma^{-1}(i)}$ explained in the proof of \Cref{prop:recurrent_inequalities}, this is precisely what the inequality $0\leq u_{\sigma^{-1}(i)}$ guarantees.

	Again referring to the proof of \Cref{prop:recurrent_inequalities} for the interpretation of the quantity $w_{\sigma^{-1}(i)}$, the inequality $u_{\sigma^{-1}(i)}< w_{\sigma^{-1}(i)}$ guarantees that, after toppling the sink,  following the toppling sequence $\sigma$ the vertex $i$ does not become unstable in a \textbf{for} loop iteration of the algorithm before when it is supposed to topple. This shows that $\sigma$ is indeed the output of \Cref{alg:toppling}, completing the proof of the proposition.
\end{proof}

The above characterization can be checked on an instance by comparing Examples~\ref{ex:u}~and~\ref{ex:w}.

\section{Bijection with parking functions}\label{sec:bijection}

We now provide a bijection between sorted recurrent configurations of $\widehat{G}_{\mu,\nu}$ and the parking functions in $\PF(\mu;\nu)$.

Let $\mathsf{SortRec}(\mu;\nu)$ be the set of sorted recurrent configurations of $\widehat{G}_{\mu,\nu}$.
We define the function \[ \Phi \colon \mathsf{SortRec}(\mu;\nu) \to \mathsf{PF}(\mu;\nu) \] in the following way: given $\kappa \in \mathsf{SortRec}(\mu;\nu)$, in the notation of \Cref{sec:SortRec}, we set $\Phi(\kappa)$ to be the (unique) parking function of size $n = \lvert \mu \rvert + \lvert \nu \rvert$ such that the label $i$ occurs in column $n-\widetilde{\kappa}(i)$ (we number the columns increasingly from left to right) for every $i \in [n]$.

Before showing that $\Phi$ is a well-defined bijection, let us look at an example.
\begin{example} \label{ex:bijection}
	The parking function $D \in \PF((4,3);(3,2))$ in \Cref{fig:LDP_config} is the image $\Phi(\kappa)$ of the configuration $\kappa$ in \Cref{ex:kappa_config} ($\widetilde{\kappa}$ is computed in \Cref{ex:u}). The colored crosses in the diagram in \Cref{fig:LDP_config} denote the extra shift coming from considering $n-\widetilde{\kappa}(i)$ instead of $n-\kappa(i)$ for $i$ in an independent component.
\end{example}

We want to show (1) that the map $\Phi$ is well-defined, i.e.\ that our definition of $\Phi(\kappa)$ actually gives a parking function in $\mathsf{PF}(\mu;\nu)$, and (2) that $\Phi$ is bijective.

For (2): let $\kappa\in \mathsf{SortRec}(\mu;\nu)$, and let $\sigma \in \mathfrak{S}_n$ be the toppling word given by \Cref{alg:toppling} applied to $\kappa$. Recall the definition of the vector $u = u(\sigma) \coloneqq (u_1,u_2,\dots,u_n)$, i.e.\ for every $i \in [n]$
\begin{equation}
	\label{eq:usigma}
	\sigma^{-1}(i)-u_{\sigma^{-1}(i)} = n-\widetilde{\kappa}(i).
\end{equation}

Hence, we can see $\Phi$ as the map sending the pair $(\sigma,u)$, which by \Cref{prop:recurrent_inequalities} satisfies $0 \leq u_i < w_i(\sigma)$ for every $i \in [n]$, into the parking function whose label $i$ occurs in column $\sigma^{-1}(i) - u_{\sigma^{-1}(i)}$ for every $i \in [n]$.

Now we can apply \cite[Theorem~44]{LoehrRemmel_pmaj}, which we restate in our notation.\footnote{To match the notation in \cite{LoehrRemmel_pmaj}, one should keep in mind that the ingredients of the occurring formulas are computed on the reverse word $\sigma(n)\sigma(n-1)\cdots \sigma(1)$, so that for example the positions $\sigma^{-1}(i)$ get complemented to $n+1$; also, the map $\Phi$ is called ``$G$'' in \cite{LoehrRemmel_pmaj}.}
\begin{theorem}[Loehr-Remmel bijection]
	\label{thm:Loehr_Remmel}
	Via the map $\Phi$ the pairs $(\sigma,u)$ with $0\leq u_j<w_j(\sigma)$ for every $j \in [n]$ are sent bijectively onto the elements of $\mathsf{PF}(n)$, in such a way that $\binom{n+1}{2}-\sum_{j=1}^n (j-u_j)$ goes into area and $\sigma$ is the pmaj word of the image, which implies that $\mathsf{maj}(\sigma_{n}\sigma_{n-1} \cdots \sigma_1)$ goes into pmaj.
\end{theorem}

For (1), it remains to observe that the parking function that we obtained is actually in $\mathsf{PF}(\mu;\nu)$: indeed notice that if $i,j \in I_{\nu_r}$ with $i<j$, then \Cref{rmk:kappatilde_increasing} implies that $\widetilde{\kappa}(i) > \widetilde{\kappa}(j)$ so that $n-\widetilde{\kappa}(i) < n-\widetilde{\kappa}(j)$, hence $i$ occurs in a column strictly to the left of the column of $j$. Therefore, $i$ occurs necessarily in a row lower than the row of $j$. Similarly, if $i,j \in K_{\mu_s}$ with $i<j$, then $n - \widetilde{\kappa}(i) \geq n - \widetilde{\kappa}(j)$, so that $i$ occurs necessarily in a row higher than the row of $j$ (inside the same column of a parking function the labels occur in increasing order from bottom to top). All this shows that the pmaj reading word of $\Phi(\kappa)$ is in $\mathsf{W}(\mu;\nu)$, as we wanted.

The preceding discussion proves the following proposition.
\begin{proposition}
	The map $\Phi$ is a well-defined bijection.
\end{proposition}

In fact \Cref{thm:Loehr_Remmel} also provides the following immediate theorem, which is the main result of this article.

\begin{theorem}
	The map $\Phi$ is a well-defined bijection such that $\area(\Phi(\kappa)) = \lev(\kappa)$ and such that the $\sigma$ obtained from \Cref{alg:toppling} applied to $\kappa$ equals the pmaj word of $\Phi(\kappa)$, so that $\pmaj(\Phi(\kappa)) = \del(\kappa)$.
\end{theorem}

\begin{proof}
	The statement about $\del$ follows from \Cref{rmk:delay}, while the statement about $\lev$ follows from the following easy computation:
	\begin{align*}
		\lev(\kappa)                          & = - \lvert E_s(\widehat{G}_{\mu,\nu}) \rvert + \sum_{i=1}^n \kappa(i)        \\
		\text{(using \Cref{rmk:level_Gmunu})} & = - \binom{n}{2} + \sum_{i \geq 0} \binom{\nu_i}{2} + \sum_{i=1}^n \kappa(i) \\
		                                      & = n^2 - \binom{n}{2} - n^2 + \sum_{i=1}^n \widetilde{\kappa}(i)              \\
		                                      & = \binom{n+1}{2} - \sum_{i=1}^n (n-\widetilde{\kappa}(i))                    \\
		\text{(using \eqref{eq:usigma})}      & = \binom{n+1}{2} - \sum_{i=1}^n (i-u_i).
	\end{align*}
\end{proof}

Now \Cref{thm:main} is an immediate consequence of this result combined with \Cref{thm:shuffle_pmaj}. It can be checked in the instance of \Cref{ex:bijection} (cf.\ Examples~\ref{ex:algo},~\ref{ex:delay},~\ref{ex:parking_word},~\ref{ex:area}~and~\ref{ex:level}).

\bibliographystyle{amsalpha}
\bibliography{bibliogr}

\end{document}